\documentclass[12pt,oneside,reqno]{amsart}
\usepackage{amsmath,amsthm,amssymb,epic,eepic,bm}
\usepackage{eqlist,eqparbox,color}
\usepackage{comment}

\title[On pseudo BL-algebras and pseudo hoops with normal maximal filters]
{On pseudo BL-algebras and pseudo hoops with normal maximal filters}
\author[Michal Botur, Anatolij Dvure\v{c}enskij]{Michal Botur$^1$, Anatolij Dvure\v{c}enskij$^{1,2}$}
\thanks{Both authors gratefully acknowledge  the support by GA\v{C}R 15-15286S.  AD thanks also Slovak Research and Development Agency under contract APVV-0178-11,  grant VEGA No. 2/0059/12 SAV}
\address{$^1$Palack\' y University Olomouc, Faculty of Sciences, t\v r. 17.listopadu 12, CZ-771 46 Olomouc, Czech Republic}
\address{$^2$Mathematical Institute, Slovak Academy of Sciences, \v{S}tef\'anikova 49, SK-814 73 Bratislava, Slovakia}
\email{michal.botur@upol.cz, dvurecen@mat.savba.sk}

\begin{document}

\newtheorem{cl}{Claim}

\newtheorem{theorem}{Theorem}[section]
\newtheorem{lemma}[theorem]{Lemma}
\newtheorem{definition}[theorem]{Definition}
\newtheorem{example}[theorem]{Example}
\newtheorem{remark}[theorem]{Remark}
\newtheorem{corollary}[theorem]{Corollary}
\newtheorem{proposition}[theorem]{Proposition}
\newtheorem{problem}[theorem]{Problem}

\newcommand{\zl}{\mbox{$[\hspace{-1.8pt} [$}}
\newcommand{\zr}{\mbox{$]\hspace{-2pt} ]$}}
\newcommand{\n}{\color{red}}
\newcommand{\m}{\color{blue}}
\newcommand{\ra}{\rightarrow}
\newcommand{\rra}{\rightsquigarrow}
\newcommand{\A}{$\mathbf A=(A;\vee,\wedge, \cdot, \rightarrow, \rra, 0,1)$}
\newcommand{\squig}{\rightsquigarrow}
\newcommand{\lex}{\,\overrightarrow{\times}\,}

\begin{abstract}
We study the class of pseudo BL-algebras whose every maximal filter is normal. We present an equational base for this class and we extend these results for the class of basic pseudo hoops with fixed strong unit.

\end{abstract}

\maketitle

\section{Introduction}

The class of $\ell$-groups plays an important role in the study of many algebraic structures like MV-algebras, BL-algebras, hoops and their non-commutative generalizations which provide both commutative and non-commutative algebraic counter parts of the \L ukasiewicz many-valued logic and of the H\'ajek fuzzy logic, \cite{Haj}.

It is well-known that the class of normal-valued $\ell$-groups is the largest proper subvariety of the variety $\mathcal L$  of $\ell$-groups, see \cite[Thm 10C]{Gla}. They are defined by the equation $a+b \le 2b + 2a$ for all $a,b\ge 0$, see \cite[Thm 41.1]{Dar}. In \cite{GeIo, Rac}, there was defined a non-commutative generalization of MV-algebras called pseudo MV-algebras, see \cite{GeIo}, or generalized MV-algebras, see \cite{Rac}. The fundamental representation of pseudo MV-algebras by intervals in the positive cone of $\ell$-groups with strong unit together with the categorical equivalence was presented in \cite{151}. Thanks to this categorical equivalence it is possible to show that every variety of $\ell$-groups gives a variety of pseudo MV-algebras. The variety of normal-valued pseudo MV-algebras was described together with an equational base in \cite{156}. However in \cite{DvHo}, it was shown that the variety of normal-valued pseudo MV-algebras is not the largest proper subvariety of the variety of pseudo MV-algebras, $\mathcal{PMV}$. In \cite{DvHo}, there was defined the class $\mathcal M$ of pseudo MV-algebras whose every maximal ideal (or filter) is normal, and this class is even a proper subvariety of the variety $\mathcal{PMV}$. And the variety of normal-valued pseudo MV-algebras is a proper subvariety of the variety $\mathcal M$. We note that we do not know yet neither any equational base for the variety $\mathcal M$ nor whether does $\mathcal{PMV}$ has the largest proper subvariety. We note that within the variety $\mathcal M$ it is possible to prove many properties that are known for MV-algebras, see e.g. \cite{DvHo, 214, DDT}. In particular every pseudo MV-algebra from $\mathcal M$ admits a state, whereas there are pseudo MV-algebras that are stateless, \cite{156}. We note that a state is an analogue of a finitely additive probability measure and it serves as averaging processes for truth-value in \L ukasiewicz logic, \cite{Mun1}.

On the other hand, the Mundici representation theorem \cite{Mun} for MV-algebras and Komori's result \cite{Kom} show that within the least non-trivial variety of $\ell$-groups, the variety of Abelian $\ell$-groups, see \cite[Thm 10.A]{Gla}, we can find countably many subvarieties of MV-algebras. The same is possible to say that using the representation theorem for pseudo MV-algebras, \cite{151}, it is possible to find among any variety of $\ell$-groups a system of more finer varieties of pseudo MV-algebras. Therefore, the study of the variety $\mathcal M$ can give an interesting view also to a special class of $\ell$-groups.

A more general class than pseudo MV-algebras is the class of pseudo BL-algebras introduced in \cite{DGI1, DGI2}. They are a non-commutative generalization of H\'ajek's BL-algebras. They form a variety that is an algebraic counterpart of fuzzy logic, \cite{Haj}. Hence, pseudo BL-algebras are an algebraic
presentation of a non-commutative generalization of fuzzy logic. In \cite{GLP} there was presented a non-commutative generalization of hoops called pseudo hoops. Such pseudo hoops do not possess the least elements, in general, and for example, the negative cone of every $\ell$-group provides an example of a pseudo hoop. We note that pseudo hoops are overlapping with notions introduced already by Bosbach as residuated integral monoids in his pioneering papers \cite{Bos1, Bos2}.

Also for pseudo BL-algebras and pseudo hoops, it is possible to define the class of normal-valued algebras and the class of algebras where every maximal filter is normal. The equational base of normal-valued basic pseudo hoops was presented in \cite{BDK}. In \cite{DvKo}, there was presented an interesting class of pseudo BL-algebras, called kite pseudo BL-algebras. A kite uses an $\ell$-group and two injections from a one set into another. According to \cite[Lem 5.1]{DvKo}, it follows that every such kite pseudo BL-algebra has a unique filter and this filter is normal. In particular, if the two injections are bijections, then the kite is a pseudo MV-algebra belonging to $\mathcal M$.

In this paper, we present an equational base for the variety of pseudo BL-algebras where every maximal filter is normal as well as for the class of basic unital pseudo hoops. This means a basic pseudo hoop with a fixed strong unit. This notion corresponds to an $\ell$-group with a fixed strong unit which represents any pseudo MV-algebra by \cite{151}.

The paper is organized as follows. Section 2 defines pseudo BL-algebras and pseudo MV-algebras and it reminds their basic notions and properties. Section 3 studies equational base of the variety of pseudo BL-algebras where every maximal filter is normal. The methods developed for pseudo BL-algebras are extended for basic pseudo hoops with strong unit which is done in Section 4.

\section{Pseudo BL-algebras and Pseudo MV-algebras}

A {\it pseudo BL-algebra}, \cite{DGI1, DGI2}, is an algebra \A\ of
type $\langle 2,2,2,2,2,0,0\rangle $ satisfying for all $x,y,z \in A$ the following conditions:

\begin{enumerate}
\item[(i)] $(A;\cdot ,1)$ is a monoid (need not be commutative),
i.e., $\cdot$ is associative with neutral element $1$;

\item[(ii)] $(A;\vee ,\wedge,0,1)$ is a bounded lattice;

\item[(iii)] $x\cdot y\leq z$ iff $x\leq y\rightarrow z$ iff $y\leq x\squig z$;

\item[(iv)]
$(x\rightarrow y)\cdot x=x\wedge y=y\cdot (y\squig x)$ (divisibility condition);

\item[(v)] $(x \to y) \vee (y\to x) = 1 = (x\squig y)\vee (y\squig y)$ (prelinearity condition).
\end{enumerate}

We recall that $\wedge, \vee$ and $\cdot$ have higher priority than $\to$ or  $\squig$, and $A$ is a distributive lattice.

We  say that a pseudo BL-algebra $\mathbf A$ is a {\it BL-algebra} if
$x\cdot y = y \cdot x$ for all $x,y \in A$.  This is equivalent to the statement $\to = \squig$.

We denote by $\mathcal{PBL}$ and $\mathcal{BL}$ the variety of pseudo BL-algebras and BL-algebras, respectively.

Let  $\mathbf A$ be a pseudo BL-algebra. Let us define two unary operations
(two negations) $^-$ and $^\sim $ on $A$ such that $x^-:=x\rightarrow 0$
and $x^\sim :=x\squig 0$ for any $x\in A$.  It is easy to show that
$$
x\cdot y = 0 \Leftrightarrow  y \le x^\sim \Leftrightarrow   x \le y^-.
$$
Moreover, $\cdot$ distributes $\vee$ and $\wedge$ from both sides.

Let $x \in A$, for any integer $n\ge 0$, we define
$$
x^0:=1, \quad x^1 :=x, \quad x^{n+1}:=x^n\cdot x\ \mbox{ for } n\ge 1.
$$

A non-empty subset $F \subseteq A$ of a pseudo BL-algebra $\mathbf A$ is said to be a {\it filter} if (i) $x,y \in F$ implies $x\cdot y \in F,$ and (ii) $x\le y\in A$ and $x \in F$ imply $y \in F.$ According to  \cite[Prop 4.7]{DGI1}, a subset $F$ of $A$ is a
filter iff (i) $1 \in F$, and (ii) $x, x\to y \in F$ implies $y\in
F$ ($x, x\squig y \in F$ implies $y\in F$), i.e., $F$ is a {\it
deductive system}.

If $a\in A,$ then the filter, $F(a),$  generated by $a$ is the set
$$
F(a)=\{x\in M: x\ge a^n\ \mbox{for some}\ n \ge 1\}.\eqno(2.1)
$$
A filter $F$ of $\mathbf A$ is (i) {\it maximal} if $F$ is a proper subset of $A$ and it is not a proper subset of any proper filter of  $\mathbf A$, (ii) {\it prime} if $x\vee y \in F$ implies $x\in F$ or $y\in F$, (iii) {\it minimal prime} if $F$ is prime and if $G$ is any prime filter of $\mathbf A$ such that $G\subseteq F$, then $G=F$,  and (iv)
{\it normal} if $x\to y \in F$ iff $x\squig y \in F$; this
is equivalent to $a\cdot F = F\cdot a$ for any $a \in A$;  here $a\cdot
F = \{a\cdot h:\ h \in F\}$ and $F\cdot a = \{h\cdot a:\ h \in F\}.$ We note that every maximal filter is prime.

We denote by $\mathcal F(\mathbf A)$, $\mathcal M(\mathbf A)$,  $\mathcal P(\mathbf A)$, $\mathcal N(\mathbf A)$, and $\mathcal{NM}(\mathbf A)$, the set of all filters, maximal filters, prime filters, normal filters, and maximal filters that are normal, respectively.

We denote by $\mathcal{MNPBL}$ the class of pseudo BL-algebras $\mathbf A$ such that either every maximal filter of $\mathbf A$ is normal or $\mathbf A$ is trivial (i.e. $A=\{1\}$). For example, every maximal filter of a linearly ordered pseudo BL-algebra is normal, see \cite[Thm 4.5]{212}; in fact, there is a unique maximal filter.
By \cite[Thm 4.1]{DGK}, \cite[Prop 4.10]{GLP}, $\mathcal{MNPBL}$ is a proper subvariety of the variety $\mathcal{PBL}$ of pseudo BL-algebras. This result for pseudo MV-algebras was firstly established in \cite[Prop 6.2]{DDT} using methods developed in \cite{DvHo}.

Let $g<1$ be an element of a pseudo BL-algebra $\mathbf A$. We say that a {\it value} of $g$ is a filter $V$ of $\mathbf A$ such that $g \notin V$, and $V$ is maximal with respect to this property; it is a prime filter. For any value $V$ of an
element $g$, there is a unique least filter $V^*$
properly containing $V$; it is called a {\it cover} of $V$. It  is
equal to the filter generated by $V$ and the element $g.$

We say that a pseudo BL-algebra $\mathbf A$ is {\it normal-valued} if every value
$V$ in $\mathbf A$ is normal in its cover $V^*$, i.e. $x,y \in V^*$ implies
$x\to y \in V$ iff $x\squig y \in V$. By \cite[Cor 4.5]{212}, every
linear pseudo hoop is normal-valued. Let $\mathcal {NVPBL}$ be the
family of normal-valued pseudo BL-algebras. By \cite[Cor 6.9]{BDK}, $\mathcal {NVPH}$ is a variety of the variety of pseudo BL-algebras that is a subvariety of the variety $\mathcal{MNPBL}$. The set of identities characterizing the variety of normal-valued pseudo BL-algebras was presented in \cite{BDK}. In \cite[Thm 4.2]{DGK} it was shown that $\mathcal{NVPBL}$ is a proper subvariety of the variety $\mathcal{MNPBL}$.

If $F$ is a normal filter, then the relation $x\sim_F y$ iff $x\to y, y\to x\in F$ (equivalently $x\rra y, y\rra x\in F$) is a congruence such that $F=\{a\in A: a\sim_F 1\}$. Conversely, for every congruence $\sim$ on $A$, there is a unique normal filter $F$ such that $\sim = \sim_F$, \cite[Prop 1.9]{DGI1}.

A {\it pseudo MV-algebra} is an algebra $\mathbf A=(A; \oplus,^-,^\sim,0,1)$
of type $\langle 2,1,1,0,0\rangle$ such that the following axioms (we use
axioms from \cite{GeIo}) hold for all $x,y,z \in A$ with an additional
binary operation $\odot$ defined via
$$ x \odot y =(x^- \oplus y^-)
^\sim
$$
\begin{enumerate}
\item[{\rm (A1)}]  $x \oplus (y \oplus z) = (x \oplus y) \oplus z$;

\item[{\rm (A2)}] $x\oplus 0 = 0 \oplus x = x$;

\item[{\rm (A3)}] $x \oplus 1 = 1 \oplus x = 1$;

\item[{\rm (A4)}] $1^\sim = 0;$ $1^- = 0$;

\item[{\rm (A5)}] $(x^- \oplus y^-)^\sim = (x^\sim \oplus y^\sim)^-$;

\item[{\rm (A6)}] $x \oplus (y \odot x^\sim) = y \oplus (x\odot y^\sim ) = ( y^-\odot x) \oplus y = ( x^-\odot y) \oplus x;$

\item[{\rm (A7)}] $(x^- \oplus y)\odot x = y\odot (x \oplus y^\sim)$;

\item[{\rm (A8)}] $(x^-)^\sim= x$.
\end{enumerate}

If we define $x \le y$ iff $x^- \oplus y=1$, then $\le$ is a partial order such that $A$ is a
distributive lattice with $x \vee y = x \oplus (y \odot x^\sim)$ and
$x \wedge y =  (x^- \oplus y)\odot x.$ For basic properties of
pseudo MV-algebras see \cite{GeIo}. We remind some of them:

\begin{enumerate}
\item[(i)]
$x^{-\sim}= x = x^{\sim-}$,
\item[(ii)] $x\odot y \le z$ iff $y \le z
\oplus x^\sim$ iff $x \le y^-\oplus z$,
\item[(iii)] $(x^-\oplus y) \vee
(y^-\oplus x) = 1 = (y \oplus x^\sim) \vee (x \oplus y^\sim).$
\end{enumerate}

We recall that a pseudo MV-algebra $\mathbf A$ is an MV-algebra iff $x\oplus y = y\oplus x$ iff $x\odot y=y\odot x$ for all $x,y \in A.$

An archetypal example of pseudo MV-algebras is obtained as follows. Let $(G,u)$ be a unital $\ell$-group written
additively, i.e. an $\ell$-group $G$ (not necessarily Abelian)\footnote{ We note that an $\ell$-{\it group} is a group $\mathbf G=(G;+,-,0,\le)$ with a neutral element $0$ and endowed with a partial order $\le$ such that (i) $a\le b$ implies $c+a+d\le c+b+d$ for all $c,d\in G$, (ii) $G$ under $\le$ is a lattice. For more info about $\ell$-groups, see e.g. \cite{Dar, Gla}.} with a fixed strong unit $u$ (i.e., given $g \in G$, there is an integer
$n\ge 1$ such that $g \le nu$). Set $\Gamma(G,u)=[0,u]=:\{g\in G:\
0\le g\le u\},$ and for all $x,y \in \Gamma(G,u)$, we define
\begin{align*}
 x\oplus y&:= (x+y)\wedge u,\\
x^-&:=u-x,\\
x^\sim &:=-x\oplus u,\\
x\odot y&:= (y-u+x)\vee 0,
\end{align*}
then $\mathbf{\Gamma}(G,u)=(\Gamma(G,u); \oplus, ^-,^\sim,0,u)$ is a pseudo MV-algebra.
According to \cite{151}, every pseudo MV-algebra is isomorphic to some
$\Gamma(G,u),$ and, in addition, the functor $\mathbf{\Gamma}$ defines a
categorical equivalence  between the category of pseudo MV-algebras
(that is a variety) and the category of unital $\ell$-groups,
$(G,u),$ that is not a variety because it is not closed under
infinite direct products.

In particular if $\mathbf G$ is an arbitrary $\ell$-group and $\mathbb Z$ is the group of integers, then the pseudo MV-algebra $\Gamma(\mathbb Z\lex \mathbf G,(1,0))$ belongs to the variety $\mathcal M$; here $\lex$ means the lexicographic product of two $\ell$-groups. We note, that for this example both negations coincide even in the case that $\mathbf G$ is not Abelian.

We note that an important class of pseudo BL-algebras is the class of pseudo MV-algebras: (1) Let $(A; \oplus, \odot, ^-,
^\sim, 0,1)$ be a pseudo MV-algebra. Define
$$
x\to_B y := x^- \oplus y, \quad x \squig_B y:= y \oplus x^\sim.
$$
Then $(A;\vee,\wedge, \odot, \to_B, \squig_B,0,1)$ is a pseudo BL-algebra such
that
$$x^{-_B} := x\to_B 0 =  x^-,\quad x^{\sim_B} := x\squig_B 0 =
x^\sim,\quad x^{-_B\sim_B} = x = x^{\sim_B-_B},
$$
and $(x^{-_B} \odot y^{-_B})^{\sim_B} = (x^{\sim_B} \odot
y^{\sim_B})^{-_B}.$

If we define $x\oplus_A y := (x^{-_B} \odot y^{-_B})^{\sim_B} =
(x^{\sim_B} \odot y^{\sim_B})^{-_B}, $ then $x\oplus_A y= x\oplus
y$,  $x^{-_A} := x^{-_B} = x^-,$ $x^{\sim_A} := x^{\sim_B} =
x^\sim$, and $(A; \oplus_A, ^{-_A},^{\sim_A},0,1)$ is a pseudo
MV-algebra that coincides with the original one.

(2) If \A\ is a
pseudo BL-algebra such that $x^{-\sim} = x = x^{\sim-}$, then
$(x^-\cdot y^-)^\sim = (x^\sim \cdot y^\sim)^-$. We set
$$ x \oplus_A y := (x^-\cdot y^-)^\sim = (x^\sim \cdot
y^\sim)^-,\quad x^{-_A} := x\squig 0,\quad x^{\sim_A} := x\to 0.
$$
Then $(A;\oplus_A,^{-_A},^{\sim_A},0,1)$ is a pseudo MV-algebra.

If we set $x\to_B y := x^{-_A} \oplus_A y$ and $x\squig_B y :=
y \oplus_A x^{\sim_A}$, then $\to_B = \to$ and $\squig_B = \squig$, and
$(A; \vee, \wedge, \cdot, \to_B,\squig_B,0,1)$ is a pseudo BL-algebra that
coincides with the original one.

In other words, a pseudo BL-algebra $\mathbf A$ is a pseudo MV-algebra if, and only if, $x^{-\sim}=x=x^{\sim-}$ for each $x \in A$.

We denote by $\mathcal{PMV}$ and $\mathcal{MV}$ the variety of pseudo MV-algebras and the variety of MV-algebras, respectively.

It is well know that there is a categorical equivalence between the variety of MV-algebras and the category of Abelian unital $\ell$-groups, \cite{Mun}. The same is true for the variety of pseudo MV-algebras and the category of unital $\ell$-groups (not necessarily Abelian), \cite{151}. 
The variety lattice of the variety $\mathcal{MV}$ is thanks to Komori \cite{Kom} countable, whereas the variety lattice of the variety $\mathcal{PMV}$ is uncountable, \cite[Thm 7.2]{DDT}, and the variety of normal-valued pseudo MV-algebras is not the greatest proper subvariety of the variety of pseudo MV-algebras, it is properly contained in the variety of pseudo MV-algebras whose every maximal filter is normal. Therefore, the study of MV-algebras and pseudo MV-algebras enables us to investigate in more details the variety of Abelian $\ell$-groups and the variety $\mathcal L$. The aim of the next two sections is to find an equational class of the variety $\mathcal{MNPBL}$ and $\mathcal M=\mathcal{MNPBL}\cap \mathcal{PMV}$, respectively.

\section{Pseudo BL-algebras where Every Maximal Filter is Normal}

In this section, we derive a countable equational base for the variety of pseudo BL-algebras whose every maximal filter is normal.

In the following two sections, we will use a shorter notion $xy$ instead of $x\cdot y$.

Let \A\ be a pseudo BL-algebra. For any filter $F\subseteq A$, we can define two equivalences $\sim_l$ and $\sim_r$ induced by $F$ by
$$x\sim_l y \,\, \mbox{if, and only if,} \,\, x\ra y, y\ra x\in F,$$
and
$$x\sim_r y \,\, \mbox{if, and only if,} \,\, x\rra y, y\rra x\in F.$$
Then the following conditions are equivalent:
\begin{itemize}
\item[(i)] The equivalence $\sim_l$ is a congruence.
\item[(ii)] The equivalence $\sim_r$ is a congruence.
\item[(iii)] The equivalences $\sim_l$ and $\sim_r$ are equal.
\end{itemize}
By $Fx$ we denote the class of equivalence $\sim_l$ containing the element $x$. Analogously, by $xF$ we denote the class of equivalence $\sim_r$ containing the element $x$. It can be shown that $Fx$=$Fy$ if, and only if, there exist $f_1,f_2\in F$ satisfying $f_1x=f_2y$ (respectively $xF$=$yF$ if, and only if, there exists $g_1,g_2\in F$ satisfying $xg_1=yg_2$) for all $x,y\in A$. In addition, $\sim_l$ induces an order $\le:=\le_{\sim_l}$ on the set of equivalences given by $Fx\le Fy$ iff $x\to y \in F$, or equivalently $fx\le y$ for some $f \in F$. Dually, the equivalence $\sim_r$ induces an order $\le:=\le_{\sim_l}$: $xF\le yF$ iff $x\rra y\in F$ iff $xg\le y$ for some $g\in F$.

In what follows, we assume that $\mathbf A$ is non-trivial, i.e. $0\ne 1$.

\begin{lemma}\label{MainLemma}
Let \A\ be a pseudo BL-algebra from $\mathcal{MNPBL}$. If there are $x_1,\dots,x_n \in A$ such that
$$
V0 = V\big (\prod_{i=1}^n x_i\big)
$$
for all maximal filters $V$ of $\mathbf A$, then
$$
0=\prod_{i=1}^n x_i^2.
$$
\end{lemma}

\begin{proof}
We are going to prove that $(\prod_{i=1}^n x_i^2)^-$ belongs to all minimal prime filters. Let us have a minimal prime filter $M$. Using Zorn's Lemma, we can see that there is a maximal filter $V$ containing  $M$, $V$ is also prime (see also \cite[Thm 4.28]{DGI1}). From our assumption $V0 = V\big (\prod_{i=1}^n x_i\big)$, we concludes that there exists minimal $j\in\{1,\dots, n\}$ satisfying $x_j\not\in V$ (otherwise, $x_i\in V$ for all $i\in\{1,\dots, n\}$ and $V0 = V\big (\prod_{i=1}^n x_i\big) =V1$ which is absurd). Using $x_1,\dots,x_{j-1}\in V$ and normality of $V$, we obtain
$$
V\big(x_j\prod_{i=j+1}^n x_i^2\big)=V1\cdot V\big(x_j\prod_{i=j+1}^n x_i^2\big)= V\big (\prod_{i=1}^{j-1} x_i\big)\cdot V\big(x_j\prod_{i=j+1}^n x_i^2\big)
$$
$$
=V\big (\big(\prod_{i=1}^{j-1} x_i\big )x_j\prod_{i=j+1}^n x_i^2\big)\leq V\big(\prod_{i=1}^{n} x_i\big ) =V0,
$$
and thus $(x_j\prod_{i=j+1}^n x_i^2\big)^-\in V.$

Because $x_j\not\in V$, then $\big(\prod_{i=1}^{j-1}x_i^2\big)x_j\not\in V$, and divisibility gives
$$
(x_j\prod_{i=j+1}^n x_i^2\big)^- \ra \big(\prod_{i=1}^{j-1}x_i^2\big)x_j\not \in V
$$
and
$$
(x_j\prod_{i=j+1}^n x_i^2\big)^- \ra \big(\prod_{i=1}^{j-1}x_i^2\big)x_j\not \in M.
$$
Since $M$ is a prime filter, using prelinearity, we obtain
$$
\big(\prod_{i=1}^{j-1}x_i^2\big)x_j \ra (x_j\prod_{i=j+1}^n x_i^2\big)^-  \in M.
$$
The exchange property, \cite[Prop 3.8(1')]{DGI1}, finally yields
$$
\big(\prod_{i=1}^{n}x_i^2\big)^-  = \big(\prod_{i=1}^{j-1}x_i^2\big)x_j \ra (x_j\prod_{i=j+1}^n x_i^2\big)^-\in M.
$$

The condition $\big(\prod_{i=1}^{n}x_i^2\big)^-\in M$ for all minimal prime filters proves that $\big(\prod_{i=1}^{n}x_i^2\big)^-=1$ which is equivalent to $\prod_{i=1}^{n}x_i^2=0$.
\end{proof}

\begin{theorem}\label{th:3.2}
Let \A\ be a pseudo BL-algebra from $\mathcal{MNPBL}$. Then the following inequalities hold for $\mathbf A$:
\begin{itemize}
\item[(i)]
$$\big(\big(\prod_{i=1}^nx_{i}\big )^-\big )^2\leq \big(\prod_{i=1}^n x_{\pi(i)}^2\big)^-
$$
for every $n\in \mathbb N$ and for every permutation $\pi$ on the set $\{1,\dots, n\}$,
\item[(ii)]
$$
(((x\ra y)^n)^-)^2\leq ((x\rra y)^{2n})^-
$$
for every $n\in \mathbb N$,
\item[(iii)]
$$
(((x\rra y)^n)^-)^2\leq ((x\ra y)^{2n})^-
$$
for every $n\in \mathbb N$.
\end{itemize}
\end{theorem}

\begin{proof}
{\rm (i)} Let $V$ be an arbitrary maximal filter. By hypotheses, $V$ is normal. Take arbitrary elements $x_1,\dots, x_n\in A$ and any permutation $\pi$ on the set $\{1,\dots, n\}$. Because $\mathbf A/V$ is a BL-algebra (it is commutative), we obtain
$$
V\big(\big(\prod_{i=1}^nx_{i}\big )^-\prod_{i=1}^nx_{\pi(i)}\big )=V\big (\big(\prod_{i=1}^nx_{i}\big )^-\prod_{i=1}^nx_{i}\big )= V0.
$$
Lemma \ref{MainLemma} yields
$$
\big (\big (\prod_{i=1}^nx_{i}\big )^-\big )^2\prod_{i=1}^nx_{\pi(i)}^2=0
$$
and consequently also
$$
\big (\big(\prod_{i=1}^nx_{i}\big )^-\big )^2\leq\big (\prod_{i=1}^nx_{\pi(i)}^2\big )^-.
$$

{\rm (ii)} Analogously to (i), let $V$ be an arbitrary maximal filter and $x,y\in A$ be arbitrary elements. Then commutativity of the algebra $\mathbf A/V$ gives
$$
V(((x\ra y)^n)^-(x\rra y)^n)=V(((x\ra y)^n)^-(x\ra y)^n)=V0
$$
for any $n\in\mathbb N$. Using Lemma \ref{MainLemma}, we obtain
$$
(((x\ra y)^n)^-)^2(x\rra y)^{2n}=0
$$
and consequently also
$$
(((x\ra y)^n)^-)^2\leq ((x\rra y)^{2n})^-.
$$

{\rm (iii)} It can be proved analogously to the proof of (ii).
\end{proof}

\begin{theorem}\label{th:3.3}
Let \A\ be a pseudo BL-algebra satisfying inequalities {\rm (i), (ii)} and {\rm (iii)} from Theorem {\rm \ref{th:3.2}}. Then $\mathbf A$ is a pseudo BL-algebra such that every  its maximal filter is normal.
\end{theorem}

\begin{proof}
Firstly we prove the following claims.
\begin{cl}\label{c1}
Let $x_1,\dots,x_n\in A$ be arbitrary elements satisfying $0=\prod_{i=1}^n x_i$. Then $0=\prod_{i=1}^n x_{\pi(i)}^2$ holds for every permutation $\pi$ on the set $\{1,\dots, n\}$.
\end{cl}

\begin{proof}
Using inequality (i), we obtain
$$
1=(0^-)^2 = \big(\big (\prod_{i=1}^n x_i\big)^-\big )^2\leq \big(\prod_{i=1}^n x_{\pi(i)}^2\big)^-\leq 1
$$
and consequently also
$0=\prod_{i=1}^n x_{\pi(i)}^2$.
\end{proof}

\begin{cl}\label{c2}
If $V$ is a maximal filter and $x\in A$ is such that $x\not\in V$, then there exist $v\in V$ and $n\in\mathbb N$ such that $vx^n=0$ $($respectively $x^nv=0)$.
\end{cl}

\begin{proof}
Because $V$ is a maximal filter, by (2.1), there exist $v_1,\dots,v_k\in V$ such that $\prod_{i=1}^k(v_ix)=0$. Using Claim \ref{c1}, we obtain $vx^{2k}=x^{2k}v=0$, where $v=\prod_{i=1}^kv_i^2\in V$.
\end{proof}
\vspace{2mm}
Let $V$ be a maximal filter and $x,y\in A$ be such that $x\ra y\not\in V.$ Using Claim \ref{c2}, there exist $v\in V$ and $n\in\mathbb N$ such that $v(x\ra y)^n =0$. Thus $v\leq ((x\ra y)^n)^-$ yields $(((x\ra y)^n)^-)^2\in V$. Inequalities (ii) entail $((x\rra y)^{2n})^-\in V$ and also $x\rra y\not\in V$ (contrary $x\rra y\in V$ yields $(x\rra y)^{2n}\in V$ and also $0=((x\rra y)^{2n})^-(x\rra y)^{2n}\in V$ which is absurd). We have proved that $x\ra y\not\in V$ implies $x\rra y\not\in A$. The converse implication can be proved analogously using inequalities (iii). Thus $V$ is a normal filter.
\end{proof}

Combining Theorems \ref{th:3.2}--\ref{th:3.3}, we have the following characterization of the variety $\mathcal{MNPBL}$:

\begin{corollary}\label{co:3.4}
A pseudo BL-algebra \A\ belongs to $\mathcal{MNPBL}$ if, and only if, it satisfies the system of inequalities {\rm (i)--(iii)} from Theorem {\rm \ref{th:3.2}}.
\end{corollary}

In the paper \cite{BDK}, it was found a countable system of inequalities that characterizes the variety of normal-valued pseudo BL-algebras, see \cite[Cor 6.9]{BDK}, which is a proper subvariety of the variety $\mathcal{MNPBL}$. If we take the variety of pseudo BL-algebras satisfying the equality $(x \to y)\rra y=(x\rra y)\to y$, this infinite system of inequalities has reduced to a unique inequality $x^2y^2\le yx$, see \cite[Thm 6.12]{BDK}.

\begin{problem}
Can we replace the system of inequalities {\rm (i), (ii)} and {\rm (iii)} by some finite one? What about pseudo MV-algebras from $\mathcal{MNPBL}$, is the inequality $(x^-)^2\leq (x^2)^-$ enough? (Or may be also $(x^{--})^2 \le ((x^\sim)^2)^-$ and $x^2 \le ((x^-)^2)^-$?)
\end{problem}

The following result enables to skip inequalities (ii) and (iii) in Theorem \ref{th:3.2}.

\begin{theorem}\label{th:3.5}
Let $\mathbf A$ be a pseudo BL-algebra satisfying the identities $(x\to y)^-=(x\rra y)^-$ and $(x\to y)(x\rra y)=(x\rra y)(x\to y)$. Then $\mathbf A \in \mathcal{MNPBL}$ if, and only if, {\rm (i)} of Theorem {\rm \ref{th:3.2}} holds for $\mathbf A$.
\end{theorem}

\begin{proof}
Let a pseudo BL-algebra $\mathbf A$ satisfy $(x\to y)^-=(x\rra y)^-$ and $(x\to y)(x\rra y)=(x\rra y)(x\to y)$. We assert that $((x\to y)^n)^- = ((x\rra y)^n)^-$ for each integer $n\ge 1$. Assume for induction that it holds for any integer $k$ with $1\le k\le n$. We have
\begin{align*}
(x\to y)^{n+1} \to 0 &= (x\to y) \to ((x\to y)^n\to 0)\quad \mbox{use \cite[Prop 3.8(1')]{DGI1}} \\
&= (x\to y)\to ((x\rra y)^n\to 0)\\
&=(x\to y)(x\rra y)^n \to 0\\
&= (x\rra y)^n(x\to y)\to 0\\
&= (x\rra y)^n \to ((x\to y)\to 0)\\
&= (x\rra y) \to ((x\rra y)\to 0)\\
&= (x\rra y)^{n+1} \to 0.
\end{align*}

To prove that $\mathbf A$ satisfies (ii) of Theorem \ref{th:3.2}, we have to show that $(((x\to y)^n)^-)^2((x\rra y)^{2n}=0$. Calculate

\begin{align*}
(((x\to y)^n)^-)^2((x\rra y)^{2n} &= ((x \to y)^n)^-((x\to y)^n\to 0)(x\rra y)^{2n}\\
&= ((x \to y)^n)^-((x\rra y)^n\to 0)(x\rra y)^{n} (x\rra y)^n\\
&= ((x \to y)^n)^-((x\rra y)^n)^-(x\rra y)^{n} (x\rra y)^n =0.
\end{align*}

Similarly, to prove that $\mathbf A$ satisfies (iii) of Theorem \ref{th:3.2}, we have to show $(((x\rra y)^n)^-)^2(x\to y)^{2n}=0.$  Check

\begin{align*}
(((x\rra y)^n)^-)^2((x\to y)^{2n} &= ((x \rra y)^n)^-((x\rra y)^n\to 0)(x\to y)^{2n}\\
&= ((x \rra y)^n)^-((x\to y)^n\to 0)(x\to y)^{n} (x\to y)^n\\
&= ((x \rra y)^n)^-((x\to y)^n)^-(x\to y)^{n} (x\to y)^n =0.
\end{align*}

\end{proof}

We note that if $\mathbf A$ is a pseudo MV-algebra, then the condition $(x\to y)^-=(x\rra y)^-$ gives $(x\to y)^{-\sim}=(x\rra y)^{-\sim}$ and $x\to y = x\rra y$, i.e. $\mathbf A$ is an MV-algebra.

On the other hand, let $\mathbf A_1$ be a linear BL-algebra and $\mathbf A_2$ be a pseudo BL-algebra that is neither a BL-algebra nor linear. Then the ordinal sum $\mathbf A= \mathbf A_1 \oplus \mathbf A_2$ is a pseudo BL-algebra that satisfies the equation $(x\to y)^-=(x\rra y)^-$ and $\mathbf A$ is not commutative.

\section{Unital Basic Pseudo Hoops where Every Maximal Filter is Normal}

In this section, we extend the results of the previous section for unital basic pseudo hoops whose each maximal filter is normal.

Pseudo BL-algebras were recently generalized in
\cite{GLP} to pseudo hoops, which were
originally introduced by Bosbach in \cite{Bos1, Bos2} under the name
residuated integral monoids.  We recall that a {\it pseudo hoop} is
an algebra $\mathbf A= (A; \cdot, \to,\squig,1)$ of type $\langle 2,2,2,0
\rangle$ such that, for all $x,y,z \in A,$

\begin{enumerate}

\item[{\rm (i)}]   $x\cdot 1 = x = 1 \cdot x;$

 \item[{\rm (ii)}] $x\to x = 1 = x\squig x;$

\item[{\rm (iii)}] $(x\cdot y) \to z = x \to (y\to z);$

 \item[{\rm (iv)}] $(x \cdot y) \squig z = y \squig
(x\squig z);$

 \item[{\rm (v)}] $(x\to y) \odot x= (y\to x)\cdot y =
x\cdot (x\squig y) = y \cdot (y \squig x).$

\end{enumerate}

If $\cdot$ is commutative (equivalently $ \to = \squig$), $\mathbf A$ is
said to be a {\it hoop}.  If we set $x \le y$ iff $x \to y=1$ (this
is equivalent to $x \squig y =1$), then $\le$ is a partial order
such that $x\wedge y = (x\to y)\odot x$ and $A$ is a
$\wedge$-semilattice. If \A\ is a pseudo hoop, then the reduct $(A; \cdot, \to,\squig,1)$ is a pseudo hoop.

We say that a pseudo hoop $\mathbf A$ is (i)  {\it bounded} if there is a least element $0,$ otherwise, $\mathbf A$ is {\it unbounded},  and (ii) {\it cancellative}, if $x\odot y=x\odot z$ and $s\odot x=t \odot x$ imply $y=z$ and $s=t.$

Let now $\mathbf G$ be an $\ell$-group (written additively and with a
neutral element $0$). On the negative cone $G^-=\{g\in G:\ g\le 0\}$
we define:  $x\cdot y :=x+y,$ $x\to
y := (y-x)\wedge 0,$ $x\squig y :=(-x+y)\wedge 0$ for $x,y
\in G^-.$  Then $(G^-;\cdot,\to,\squig,0)$ is an unbounded (whenever $G\ne \{0\}$) cancellative pseudo hoop.

A pseudo hoop $\mathbf A$ is said to be {\it basic} if, for all $x,y,z \in
A,$

\begin{enumerate}

\item[{\rm (B1)}] $(x\to y) \to z \le ((y\to x)\to z)\to z$;

 \item[{\rm (B2)}] $(x\squig y) \squig z \le ((y\squig
x)\squig z)\squig z$.

\end{enumerate}

Every basic pseudo hoop is a distributive lattice, and for every pseudo BL-algebra \A, the reduct $(A; \cdot, \to,\squig,1)$ is a basic pseudo hoop.

For example, if $\mathbf A= \mathbf A_1 \oplus \{0,1\}$ is the ordinal sum of a pseudo hoop $\mathbf A_1$ that is not linearly ordered and of a two-element Boolean algebra $\{0,1\}$, then $\mathbf A$ is a pseudo hoop that is not basic \cite[Rem 5.10]{GLP}.

We note that a pseudo hoop is basic if, and only if, the prelinearity holds in it, \cite[Lem 4.5,4.6]{GLP}, \cite[Prop 3.2]{BDK}.

For pseudo hoops we define filters, maximal filters, normal filters, and values, respectively, in the same way as those for pseudo BL-algebras. We define prime filters  and minimal prime filters only for basic pseudo hoops and also in the same way as those for pseudo BL-algebras. The basic properties of prime and minimal prime filters of basic pseudo hoops can be viewed in \cite[Sect 4]{BDK}.

We say that an element $u \in A$ is a {\it strong unit} of a pseudo hoop $\mathbf A$ if the filter generated by $u$ coincides with $A$, and similarly as for $\ell$-groups, the couple $(\mathbf A,u)$, where $\mathbf A$ is a pseudo hoop with a fixed strong unit $u$, is said to be a {\it unital pseudo hoop}. For example, if $\mathbf A$ is a pseudo BL-algebra, then the bottom element $0$ is a strong unit. It is important to notice that the class of unital pseudo hoops is closed under subalgebras, homomorphic images, but, similarly as the class of unital $\ell$-groups, it is not closed under infinite direct products.

On the other hand, the class of cancellative unital pseudo hoops is categorically equivalent to the class of unital $\ell$-groups and this class in view of \cite{151} is categorically equivalent to the variety of pseudo MV-algebras.

We note that a maximal filter does not exist in any pseudo hoop. For example, if $A=(0,1]$ is a real interval equipped with $s\cdot t = \min\{s,t\}= s\wedge t$, $s\vee t=\max\{s,t\}$, and $s\to t = 1$ iff $s\le t$, otherwise $s\to t = t$
$(s,t \in (0,1])$, then the hoop $A$ does not posses any maximal filter; all filters are of the form $F_s:=(s,1]$, $s \in (0,1]$.

\begin{lemma}\label{le:4.1}
Let $\mathbf A$ be a pseudo hoop.

{\rm (1)} If $\mathbf A$ has a strong unit, then it has a maximal filter.

{\rm (2)} If $\mathbf A$ is linearly ordered, then it has a maximal filter if, and only if, $\mathbf A$ has a strong unit. In any such case, $\mathbf A$ has a unique maximal filter, $V$, and every element $u \in A\setminus V$ is a strong unit.

{\rm (3)} If $\mathbf A$ contains a proper filter $V$ containing all proper filters of $\mathbf A$, then every element $u \in A\setminus V$ is a strong unit and $V$ is a maximal filter.

{\rm (4)} If every proper filter of $\mathbf A$ belongs to some maximal filter and $\bigcup\{V: V \in \mathcal M(\mathbf A)\}\ne A$, then every element $u \in A\setminus \bigcup\{V: V \in \mathcal M(\mathbf A)\}$ is a strong unit.
\end{lemma}

\begin{proof}
(1) Let $u$ be a strong unit of $\mathbf A$. Then any value of $u$ is a maximal filter of $\mathbf A$.

(2) Now let $\mathbf A$ be linearly ordered and let $V$ be a maximal filter of $\mathbf A$. Since $A$ is linearly ordered, every proper filter of $\mathbf A$ is contained in $V.$ Then every element $u \in A\setminus V$ is a strong unit of $\mathbf A$.

(3) Let $\mathbf A$ be a pseudo hoop with a proper filter $V$ containing all proper filters of $\mathbf A$. Then $V$ is maximal. If $u \in A\setminus V$, then $F(u)$ cannot be a proper filter, proving $u$ is a strong unit.

(4) The proof follows the same steps as that of (3).
\end{proof}

We note that if $\mathcal{MNPH}$ is the class of pseudo hoops $\mathbf A$ such that either every maximal filter of $\mathbf A$ is normal or $A=\{0\}$.  In contrast to the class of $\mathcal{MNPBL}$, the class $\mathcal{MNPH}$ is not a variety, see \cite[Rem 1]{DGK}. In what follows we show when a unital basic pseudo hoop belongs to the class $\mathcal{MNPH}$. To show that, we will follow ideas developed in the previous section and we adopt the corresponding proofs.

\begin{lemma}\label{le:4.2}
Let $(\mathbf A,u)$ be a unital basic pseudo hoop from $\mathcal{MNPH}$. If there are $x_1,\dots,x_n \in A$ such that
$$
V\big (\prod_{i=1}^n x_i\big)\le Vu
$$
for all maximal filters $V$ of $\mathbf A$, then
$$
\prod_{i=1}^n x_i^2\le u.
$$
\end{lemma}

\begin{proof}
We follow the basic steps from the proof of Lemma \ref{MainLemma}.

We show $(\prod_{i=1}^n x_i^2)\to u$ belongs to all minimal prime filters. Let $M$ be a minimal prime filter, and let $V$ be an arbitrary maximal filter containing $M$.

From our assumption $Vu = V\big (\prod_{i=1}^n x_i\big)$ we concludes that there exists minimal $j\in\{1,\dots, n\}$ satisfying $x_j\not\in V$ (otherwise, $x_i\in V$ for all $i\in\{1,\dots, n\}$ and $Vu = V\big (\prod_{i=1}^n x_i\big) =V1$ which is absurd). Using $x_1,\dots,x_{j-1}\in V$ and the normality of $V$, we obtain
$$
V\big(x_j\prod_{i=j+1}^n x_i^2\big)=V1\cdot V\big(x_j\prod_{i=j+1}^n x_i^2\big)= V\big (\prod_{i=1}^{j-1} x_i\big)\cdot V\big(x_j\prod_{i=j+1}^n x_i^2\big)
$$
$$=V\big (\big(\prod_{i=1}^{j-1} x_i\big )x_j\prod_{i=j+1}^n x_i^2\big)\leq V\big(\prod_{i=1}^{n} x_i\big ) \le Vu,
$$
and thus $(x_j\prod_{i=j+1}^n x_i^2\big)\to u\in V.$

Because $x_j\not\in V$, then $\big(\prod_{i=1}^{j-1}x_i^2\big)x_j\not\in V$, and divisibility gives
$$
((x_j\prod_{i=j+1}^n x_i^2\big)\to u) \ra \big(\prod_{i=1}^{j-1}x_i^2\big)x_j\not \in V
$$
and
$$
((x_j\prod_{i=j+1}^n x_i^2\big)\to u) \ra \big(\prod_{i=1}^{j-1}x_i^2\big)x_j\not \in M.
$$
Since $M$ is a prime filter, using prelinearity, we obtain
$$
\big(\prod_{i=1}^{j-1}x_i^2\big)x_j \ra ((x_j\prod_{i=j+1}^n x_i^2\big)\to u) \in M.
$$
The exchange property, \cite[Thm 2.2(A2)]{GLP}, finally yields
$$
\big(\prod_{i=1}^{n}x_i^2\big)\to u  = \big(\prod_{i=1}^{j-1}x_i^2\big)x_j \ra ((x_j\prod_{i=j+1}^n x_i^2\big)\to u)\in M.
$$

The condition $\big(\prod_{i=1}^{n}x_i^2\big)\to u\in M$ for all minimal prime filters proves that $\big(\prod_{i=1}^{n}x_i^2\big)\to u=1$ which is equivalent to $\prod_{i=1}^{n}x_i^2\le u$.
\end{proof}

\begin{theorem}\label{th:4.3}
Let $(\mathbf A,u)$  be a unital basic pseudo hoop from $\mathcal{MNPH}$. Then the following inequalities hold for $\mathbf A$:
\begin{itemize}
\item[(i)]
$$\big(\big(\prod_{i=1}^nx_{i}\big )\to u\big )^2\leq \big(\prod_{i=1}^n x_{\pi(i)}^2\big)\to u
$$
for every $n\in \mathbb N$ and every permutation $\pi$ on the set $\{1,\dots, n\}$,
\item[(ii)]
$$
(((x\ra y)^n)\to u)^2\leq ((x\rra y)^{2n})\to u
$$
for every $n\in \mathbb N$,
\item[(iii)]
$$
(((x\rra y)^n)\to u)^2\leq ((x\ra y)^{2n})\to u
$$
for every $n\in \mathbb N$.
\end{itemize}
\end{theorem}

\begin{proof}
{\rm (i)} Let $V$ be an arbitrary maximal filter. By hypotheses, $V$ is normal. Take arbitrary elements $x_1,\dots, x_n\in A$ and any permutation $\pi$ on the set $\{1,\dots, n\}$. Because $\mathbf A/V$ is a hoop (it is commutative), we obtain
$$
V\big(\big(\big(\prod_{i=1}^nx_{i}\big )\to u\big)\prod_{i=1}^nx_{\pi(i)}\big )=V\big (\big(\big(\prod_{i=1}^nx_{i}\big )\to u\big)\prod_{i=1}^nx_{i}\big )= V\big(\big(\prod_{i=1}^nx_{i}\big) \wedge u\big) \le  Vu.
$$
Lemma \ref{le:4.2} yields
$$
\big (\big (\big(\prod_{i=1}^nx_{i}\big )\to u \big)\big )^2\prod_{i=1}^nx_{\pi(i)}^2\le u
$$
and consequently also
$$
\big (\big(\prod_{i=1}^nx_{i}\big )\to u\big )^2\leq\big (\prod_{i=1}^nx_{\pi(i)}^2\big )\to u.
$$

{\rm (ii)} Analogously to (i), let $V$ be an arbitrary maximal filter and $x,y\in A$ be arbitrary elements. Then commutativity of the algebra $\mathbf A/V$ gives
$$
V(((x\ra y)^n\to u)(x\rra y)^n)=V(((x\ra y)^n)\to u)(x\ra y)^n)= V((x\ra y)^n\wedge u)\le Vu
$$
for any $n\in\mathbb N$. Using Lemma \ref{le:4.2}, we obtain
$$
(((x\ra y)^n)\to u)^2(x\rra y)^{2n}\le u
$$
and consequently also
$$
(((x\ra y)^n)\to u)^2\leq ((x\rra y)^{2n})\to u.
$$

{\rm (iii)} It can be proved analogously to the proof of (ii).
\end{proof}

\begin{theorem}\label{th:4.4}
Let $(\mathbf A,u)$ be a unital basic pseudo hoop satisfying inequalities {\rm (i), (ii)} and {\rm (iii)} from Theorem {\rm \ref{th:4.3}}. Then $\mathbf A$ is a pseudo hoop  belonging to $\mathcal{MHPH}$.
\end{theorem}

\begin{proof}
Firstly we prove the following claims.
\begin{cl}\label{c3}
Let $x_1,\dots,x_n\in A$ be arbitrary elements satisfying $\prod_{i=1}^n x_i\le u$. Then $\prod_{i=1}^n x_{\pi(i)}^2\le u$ holds for every permutation $\pi$ on the set $\{1,\dots, n\}$.
\end{cl}

\begin{proof}
Using inequality (i), we obtain
$$
1=(u\to u)^2 \leq \big(\big (\prod_{i=1}^n x_i\big)\to u\big )^2\leq \big(\prod_{i=1}^n x_{\pi(i)}^2\big)\to u\leq 1
$$
and consequently also
$\prod_{i=1}^n x_{\pi(i)}^2\le u$.
\end{proof}

\begin{cl}\label{c4}
If $V$ is a maximal filter and $x\in A$ is such that $x\not\in V$, then there exist $v\in V$ and $n\in\mathbb N$ such that $vx^n\le u$ $($respectively $x^nv\le u)$.
\end{cl}

\begin{proof}
Because $V$ is a maximal filter, by (2.1), there exist $v_1,\dots,v_k\in V$ such that $\prod_{i=1}^k(v_ix)\leq u$. Using Claim \ref{c3}, we obtain $vx^{2k}=x^{2k}v\le u$, where $v=\prod_{i=1}^kv_i^2\in V.$
\end{proof}
\vspace{2mm}
Let $V$ be a maximal filter and $x,y\in A$ be such that $x\ra y\not\in V$. Using Claim \ref{c4}, there exist $v\in V$ and $n\in\mathbb N$ such that $v(x\ra y)^n \le u$. Thus $v\leq ((x\ra y)^n)\to u$ yields $(((x\ra y)^n)\to u)^2\in V$. Inequalities (ii) entail $((x\rra y)^{2n})\to u\in V$ and also $x\rra y\not\in V$ (contrary $x\rra y\in V$ yields $(x\rra y)^{2n}\in V$ and also $u \ge (((x\rra y)^{2n})\to u)(x\rra y)^{2n}\in V$ which is absurd). We have proved that $x\ra y\not\in V$ implies $x\rra y\not\in A$. The converse implication can be proved analogously using inequalities (iii). Thus $V$ is a normal filter.
\end{proof}

Using two last theorems, we have the following criterion:

\begin{theorem}\label{th:4.5}
A unital basic pseudo hoop $(\mathbf A,u)$ belongs to $\mathcal{MNPH}$ if, and only if, it satisfies the system of inequalities {\rm (i)--(iii)} from Theorem {\rm \ref{th:4.3}}.
\end{theorem}

\end{document}